\documentclass[a4paper, 12pt]{article}
\usepackage{graphicx}
\usepackage{amssymb,amsfonts,amsthm, amscd}
\usepackage{amsmath}
\usepackage{hyperref}
\usepackage[mathcal]{eucal}

\usepackage[left=3cm,right=3cm,
    top=3cm,bottom=3cm,bindingoffset=0cm]{geometry}

\newtheorem{theorem}{{\sc Theorem}}[section]

\newtheorem{lemma}{{\sc Lemma}}[section]

\newtheorem{definition}{{\sc Definition}}[section]

\begin{document}

\author{Tigran Hakobyan}
\title{\textbf {On the unboundedness of common divisors of distinct terms of the sequence $a_n=2^{2^n}+d$ for $d>1$}} 
 \date{}
\maketitle

\begin{flushleft}
\textbf{\ \ \  Abstract.} \ It is well-known that for any distinct positive integers $k$ and $n$, the numbers $2^{2^k}+1$ and $2^{2^n}+1$ are relatively prime. In this paper we consider the situation when 1 is replaced by some positive integer $d>1$.
\end{flushleft}
\subsection{Introduction}
Consider the sequence $a_n=2^{2^n}+1, n\in \mathbb N$. Since $a_n=\prod_{i=0}^{n-1}a_i+2$, we therefore have that gcd$(a_n,a_m)=1$ 
for any unequal positive integers $m$ and $n$. In this respect a natural question arises: is this propery preserved if we replace 1 by some positive integer $d>1$. 
The main result of this paper is the following theorem

\begin{theorem}
For any positive integers $d>1$ and $m$ there exist two distinct elements $a_k$ and $a_l$ in the sequence $a_n=2^{2^n}+d, n\in \mathbb N$, such that  gcd$(a_k,a_l)>m.$ 
\end{theorem}
\subsection{Proof of the theorem}
\begin{definition}
We define $\nu_{p}(m)=\max\{k:p^k|m\}$ for any integer $m$ and prime number $p$.
\end{definition}
\begin{proof}
Suppose to the contrary that there are positive integers $d>1$ and $m$ such that gcd$(a_k,a_l)\leq m$ for any distinct $k$ and $l$. It follows that if for some $t,k,l$ ($k\neq l$) $p^t|a_k$ and $p^t|a_l$ then $p^t\leq $gcd$(a_k,a_l)\leq m$, which shows that for any prime $p$ the sequence $(\nu_p(a_n))_{n\in\mathbb N}$ is bounded. Let us now prove several lemmas. \newline
\begin{lemma}
For any positive integers $n$ and $k$ satisfying $\nu_{2}(k)<n$ there exists some positive integer $l>n$ such that $(2^l-2^n)\vdots k$
(here $m\vdots n$ means $n$ divides $m$).
\end{lemma}
\begin{proof}
Suppose $k=2^a b$, where $a<n$ and $b$ is odd. Since $b|2^{\phi(b)}-1$ and $a<n$, we infer that $(2^{n+\phi(b)}-2^n)\vdots 2^n b\vdots 2^a b=k$. The lemma is proved.
\end{proof}
\begin{lemma}
If for some prime $p>m$ and positive integer $n$ we have that $p|a_n$, then $p\equiv 1(mod \ 2^n)$.
\end{lemma}
\begin{proof}
Suppose $p\neq 1(mod \ 2^n)$. By lemma 1 there exists some $l>n$, such that $(2^l-2^n)\vdots (p-1)$. Hence $a_l-a_n=2^{2^l}-2^{2^n}=2^{2^n}(2^{2^l-2^n}-1)\vdots 2^{2^n}(2^{p-1}-1)\vdots p$. Since $p|a_n$, we can infer that $p|a_{l}$, which implies $p\leq(a_n,a_{l})\leq m$, a contradiction to the conditions of the lemma. The lemma is proved.
\end{proof}
\begin{lemma}
 $d$ is a power of 2.
\end{lemma}

\begin{proof}
For any positive integer $n$, let $a_n=2^{k_n}b_n c_n$, where $b_n$ contains only odd prime divisors of $a_n$, which are less than $m$. If there is no such prime divisor we define $b_n=1$. It follows from lemma 2 that $c_n\equiv 1(mod \ 2^n)$ hence $a_n\equiv2^{k_n}b_n (mod \ 2^n)$. On the other hand $a_n\equiv d (mod \ 2^n)$, therefore $2^{k_n}b_n\equiv d (mod \ 2^n)$. Since the number of primes less than $m$ is finite and  $(\nu_{p}(a_n))_{n\in N}$ is bounded for any prime $p$, there is a positive integer $M$, such that $2^{k_n}b_n\leq M$ for any positive integer $n$. In this way we get that $2^{k_n}b_n=d$ for sufficiently large $n$. From this we infer that $d|a_n=2^{2^n}+d$ and so $d|2^{2^n}$, which exactly means  that $d$ is a power of 2. The lemma is proved.
\end{proof}

\begin{lemma}
For sufficiently large positive integer $n$ there exists a positive integer  $l>n$, such that $a_l\vdots a_n$.
\end{lemma}

\begin{proof}
According to lemma 3, $d=2^k$, for some $k$. Let us choose some $n>\nu_{2}(k)$, then $\nu_{2}(2^n-k)=\nu_{2}(k)$. Select some $l>n$ from lemma 1, then $(2^l-2^n)\vdots (2^n-k)$, so $(2^l-k)\vdots (2^n-k)$. As 
$\nu_{2}(2^l-k)=\nu_{2}(k)=\nu_{2}(2^n-k)$, the fraction $\frac{2^l-k}{2^n-k}$ is an odd number, which shows that 
$(2^{2^l-k}+1)\vdots (2^{2^n-k}+1)$, thereby after multiplying by $2^k$ we will get that $a_l\vdots a_n$. The lemma is proved.
\end{proof}
From lemma 4 it follows that $a_n=(a_n,a_l)\leq m$ for $n>\nu_2(k)$, which is a contradiction. So, our assumption was wrong and the theorem is now proved.
\end{proof}

\end{document}